\documentclass[12pt]{article}

\usepackage{amsmath}
\usepackage{amsfonts}
\usepackage{amssymb}
\usepackage{amsthm}
\usepackage{wrapfig}
\usepackage{caption}

\newcommand{\f}{\varphi}
\renewcommand{\gg}{\mathbf{g}}
\renewcommand{\ss}{\mathbf{s}}
\newcommand{\hh}{\mathbf{h}}

\newcommand{\bC}{\mathbf{C}}

\newcommand{\id}{\operatorname{id}}

\newcommand{\sprt}{\operatorname{sprt}}
\newcommand{\dep}[1]{\,=\!\!(#1)}

\newcommand{\Z}{\mathbb{Z}} 
 
\newcommand{\R}{\mathbb{R}}

\usepackage[text={16.5cm,22.5cm}]{geometry}

\newcommand{\II}{\mathcal{I}}
\newcommand{\XX}{\mathcal{X}}
\newcommand{\YY}{\mathcal{Y}}
\newcommand{\ZZ}{\mathcal{Z}}

\newcommand{\rest}{\!\!\restriction\!\!}

\newtheorem{Thm}{Theorem}
\newtheorem{Lemma}{Lemma}

\theoremstyle{definition}

\newtheorem{Def}{Definition}

\newtheorem{Notation}{Notation}

\usepackage[framemethod=TikZ]{mdframed}

\newsavebox\mybox

\newmdenv[linecolor=black, outerlinewidth=2, roundcorner=10pt,
backgroundcolor=gray!5]{exclmbox}
\newmdenv[linecolor=black, outerlinewidth=2, roundcorner=10pt,
backgroundcolor=gray!5]{exmplbox}

\newcommand{\MyBox}[1]{%
  \begin{exclmbox}%
    #1
  \end{exclmbox}%
}

\title{$G$-systems and 4E Cognitive Science}
\author{Vadim K. Weinstein}

\begin{document}

\maketitle

\begin{abstract}
  We introduce a class dynamical systems called $G$-systems equipped
  with a coupling operation. We use $G$-systems to define the notions
  of dependence (borrowed from dependence logic) and causality
  (borrowed from Pearl) for dynamical systems. As a converse to
  coupling we define decomposition or ``reducibility''. We give a
  characterization of reducibility in terms of the dependence
  ``atom''. We do all this with the motivation of developing
  mathematical foundations for 4E cognitive science, see introductory
  sections.
\end{abstract}

\section{Philosophical motivation}

Classical cognitive science views the biological body as the
``hardware'' and cognitive processes as a ``software'' which runs on
it. This approach overlooks the deep and intricate entanglement
between the two and the constraints they impose on each
other. Enactivist, embodied, extended, and embedded (4E) approaches to
understanding cognition address this gap.  These theories argue that
cognition is strongly shaped and constrained in non-trivial ways by
the nature of the sensorimotor interactions an agent engages in, or
can potentially engage in
\cite{Gibson1979,Noe2004,HuttoMyin2013,Beer2023}.

\subsection{Example: Tetrapus}

\begin{figure}[t!]
  \centering
  \includegraphics[width=\textwidth]{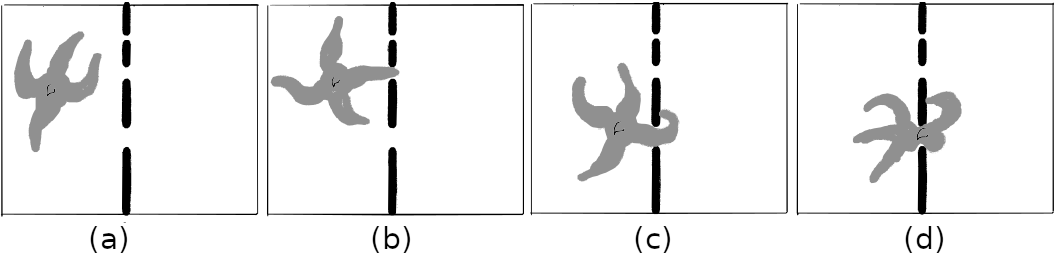}
  \caption{\footnotesize (a) Tetrapus swimming in a container. (b) Palpating the
    wall. Occationally entering holes with a tentacle. (c) If the
    tentacle can bend on the other side, it triggers a pulling
    behaviour. This bending can only happen if the tentacle can go
    deep enough into the hole. (d) Tetrapus ends up pulling itself
    through the hole. Since the thickness of the tentacle matches the
    size of the beak, the hole is of the right diameter and the pull
    is successful. }
  \label{fig:Tetrapus}
\end{figure}

It is known that due to its softness an octopus can fit through a very
narrow opening. As long as the opening is just slightly wider than its
beak, it can fit through because the beak is the only hard part of its
body. In our illustration (Figure~\ref{fig:Tetrapus}) we present a
\emph{tetrapus} (because it is easier to draw). Its tentacles' middle
part, when not squeezed, has the same thickness as the beak. The
tetrapus tries to get to the other side of the wall. It palpates its
surroundings with its tentacles. If a tentacle goes into a hole and
penetrates it so deeply that it can bend on the other side this can be
sensed by the tetrapus and triggers the behaviour of pulling the
entire body through that hole. The bending is possible precisely when
the tentacle is just so deep in the hole that the thick part is
through the hole.

The specific tentacle thickness in conjuction with the appropriate
policies of the tetrapus enable it to be ``fit'' to the environment
and the tasks that are important possibly for its survaval.  Such
fitness is described as \emph{attunement} by the 4E approahces.  

One may be tempted to describe this by saying that
when the tentacle can bend, the octopus ``knows'' that it can fit
through the hole.  But ``knowledge'' is not defined in our model. All
that is defined are processes that are triggered by other processes
modulated by sensory feedback. To ascribe ``knowledge'' of ``being
able to fit through the hole'' to the tetrapus neglects the fact that
there is very little going on in its internal
system. ``Being able to fit through the hole'' is a complex statement
about the geometry of the vast state space which consists of all
possible configurations of the mushy body of the tetrapus in the
given minimal enviroment. Instead, this can be seen as ``outsourced''
by the agent into the geometry of its own tentacles. One may also
conceive of more complex scenarios where the tentacle is less thick
than the beak, so apart from fitting in through the hole, the
tetrapus would need to wiggle the tentacle in it before the pulling
behaviour is triggered. Moreover, in another scenario the pulling
behaviour may not be triggered even when the hole has already been
found. The tetrapus would keep its tentacle in the hole while feeding
on the floating delicacies in the left chamber.  The smell of a
predator would then trigger the pulling behaviour.  Note how the
positioning of the tentacle in the hole replaces ``memory'' for the
agent. These embodied processes are regulared by internal dynamics,
but they work only when these dynamics are coupled with the
appropriate type of environment.
Enactivists often use the term
\emph{attunement} to describe the ``fit'' between the agent's
brain-body system with its environment.

\subsection{A brief introduction to 4E}

Since the publication of \emph{The Embodied
  Mind} \cite{Varela1992} where the term \emph{enactivism} was coined,
4E cognition has been developing into a new paradigm within cognitive
science and philosophy of mind.  It has given rise to many new
concepts to characterize cognition, and new, non-representationalist,
ways of explaining it
in terms of body-environment coupled dynamics, sensorimotor
contingencies, and autonomous attunement to environmental affordances
\cite{van_Gelder1998,ORegan2001,Noe2004,SML2017,Gallagher2017,Beer2023}.
Today \emph{4E cognition} stands for \emph{enactive}, \emph{embedded},
\emph{emergent}, \emph{extended}, \emph{embodied}, and
\emph{ecological}. Similar ideas have been (partially independently)
discovered also in psychology, robotics and neuroscience
\cite{Gibson1979,Brooks1991, PrinzBarsalou2000, Kauffman2000,
  Jordanous2020}. The basic contrast to classical cognitive science is
that enactivism rejects representations as axiomatic in explaining
cognition and proposes instead that the fundamental building blocks
are the processes of regulating sensitivity, motor responce adaptation,
and sensorimotor policy updates. Colloquially we refer
to basic concepts in 4E cognition by \emph{4E concepts} (see Box~1).
All these approaches have in common that cognition
is viewed as the result of sensorimotor coupling between the agent and
the environment.

\MyBox{\textbf{Box 1.}  Some of the key concepts of 4E cognition, or
  \textbf{4E concepts} are: \textbf{autonomy} (the agent sets up its
  own goals), \textbf{attunement} (sensitivity to relevant features),
  \textbf{affordances} (possibilities to engage in behaviours),
  \textbf{emergence} (high-level properties of the agent-environment
  system irreducible to low-level properties), \textbf{situatedness}
  (dependency of cognitive function on the characteristics of
  immediate environment), \textbf{robustness} (the ability to maintain
  a grip / control over the situation) \textbf{relevance} (the ability
  to pay attention to a few select features and ignoring others), and
  \textbf{openness} (the ability to be selectively alert and react to
  unexpected stimuli such as accidentally meeting a friend on the
  street, or the ringing of the phone)}


Philosophers and cognitive scientists have been proposing already for
decades to develop a DST-based formalism for E-cognition. Such a
formalism is expected to ground, unify, and clarify concepts of
enactive and embodied cognitive science as well as create bridges from
it to robotics, AI, and more
\cite{Harvey,van_Gelder1998,Beer2014b,Buhrmann2013,SML2017,Hipolito2022a,
  Newman2024}. So far, the concepts of DST have only been
applied heuristically to talk about and explain 4E concepts 
while a rigorous mathematical theory is lacking.  Using DST concepts
as metaphors to describe 4E concepts
is useful but it does not do the heavy lifting nor grant the true
benefits of an actual formalization..

\subsection{Technical motivation}
\label{ssec:Classical}

\vspace{-1\baselineskip}
\begin{wrapfigure}[12]{R}{0.3\textwidth}
  \begin{tikzpicture}[
    node style/.style={draw, circle, minimum size=1cm, font=\sffamily\Large},
    edge style/.style={draw, -stealth, thick}
  ]
  \node[node style] (Y) at (0, 1) {$Y$}; 
  \node[node style] (I) at (1.6, 0) {$I$};
  \node[node style] (U) at (0, -1) {$U$};
  \node[node style] (X) at (-1.6, 0) {$X$}; 
  
  \draw[->, edge style, bend left=20] (Y) to node[midway, above] {$\varphi$} (I);
  \draw[->, edge style, bend left=20] (U) to node[midway, below] {$f$} (X);
  \draw[->, edge style, bend left=20] (X) to node[midway, above] {$h$} (Y);
  \draw[->, edge style, bend left=20] (I) to node[midway, below] {$\pi$} (U);
\end{tikzpicture}
\caption{Agent-environment dynamics: The classical model}
\label{fig:Transition}
\end{wrapfigure}
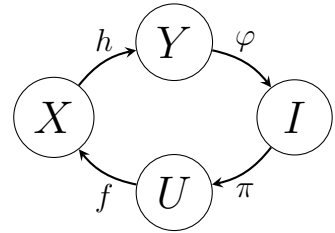

\phantom{a} \newline \textbf{The classical model.} Many authors in
robotics and 4E cognition
\cite{Weinstein_FNR,Weinstein_EMS,Sakcak2024,Weinstein_WAFR,Weinstein_CCN}
use the following formulation of agent-environment dynamics.

The \emph{external state space} is a tuple $(X,U,Y,h,f)$, where the
set $X$ is the state space or \emph{configuration space}. Its points
correspond to all possible configurations of the agent's body in the
environment.  For example, if the agent has the shape of a rod of
length $1$ and the environment is $E\subset\R^2$, then $X$ is the set
of all isometric embeddings of $[0,1]$ into $E$. The set $U$ is the
set of possible motor commands the agent can execute, while
$f\colon X\times U \to X$ is the \emph{transition function},
describing how the configuration evolves with a given motor
command. $Y$ is the set of all possible sensor readings, and
$h\colon X\to Y$ is the \emph{sensor mapping}. The \emph{internal
  state space} is a tuple $(I,\f,U,Y\pi)$, where $I$ is the set of
internal states of the agent, $\f\colon I\times Y\to I$ is the
\emph{internal transition function} that updates the internal state
based on sensor readings, and $\pi\colon I\to U$ is the agent's
\emph{policy}. The coupled system evolves from
$(\iota,x)\in I\times X$ by computing the external state
$x'=f(x,\pi(\iota))$ and then updating the internal state
$\iota'=\f(\iota,h(x'))$.  This model is standard and upto a change in
notation is the same in control theory, reinforcement learning,
minimal models of cognition, as well as the framework of the free
energy principle.

\textbf{Strength of the classical model.} This model is very abstract
and therefore is good for carrying out qualitative analysis of the
properties of agent-environment dynamics. It is good for formulating
abstract theorems and principles governing such coupling such
distinguishability of environments \cite{Weinstein_EMS}. It was used
to begin building a mathematical theory of 4E cognition
\cite{Weinstein_FNR}. $G$-systems retain these strengths while
avoiding the pitfalls describe below.

\textbf{Weaknesses of the classical model.} The state space $X$, even
though it is called the \emph{external state space}, heavily depends
on the embodiment of the agent, as do $f$ and $h$. If the agent
``grows'' a new limb, for example, the new state space $X'$ has little
to do with $X$. In the above example, imagine that the agent is no
longer a rod like $[0,1]$ but constists of two rods which are attached
at a common point which forms a movable joint. Then $U$ changes (now
there is a new motor action controlling the joint), $X$ changes (the
set of embeddings is different), and as a consequence, the domain or
range of $f$, $h$ and $\pi$ change. The model doesn't say anything
about how the new external state space relates to the old one. If the
same agent is placed in another environment, there is no way of
determining the new sensor mapping given the old one. This does not
reflect our intuition. Suppose the sensor is an eye or a camera. It's
functionality is predictable in a new environment as it reacts to the
flow of photons in a predetermined way no matter where it is, and it
will not react to, say, sounds.  Thus, ironically, even though 4E
cognition insists on a strong coupling between the body and the
environment, we are faced with the problem of \emph{separating} the
body from the environment in the mathematical realm. The point is not
philosophical but rather technical. We should be able to put the same
agent in different envrionments or the modification of the same agent
into the same environemnt and keep track of how the state coupled
systems change.

Another drawback is the underspecification of the granular structure
of both the agent and the environment. Thus, the model is
\emph{too abstract}. Such specification is necessary to be able to talk about
subsystems, modularity (visual system, auditive system, higher
cognitive processes etc.), and eventually 4E concepts such as
attunement, affordances, and the like (see Box~1). For example, we
might want to say that some \emph{component} of the internal state
space is sensitive to some \emph{aspect} of the environment while some
other component is not sensitive to the same aspect. We want to have a
systematic way of saying that $h$ (the sensor mapping) detects
this-and-this \emph{property} of the environment but is ``blind'' to
another property.  Every mathematical model abstracts away some
features and retains others. From the 4E perspective, this particular
model selects the ``wrong'' set of featues.

\textbf{Example.} A simple example illustrates how we can overcome
these problems.  Conaider a set of $n$ spherical rigid bodies in
$\R^3$ equipped with Newtonian mechanics.  A state of the system is
given by $e=(p_i,m_i,v_i,r_i,F_i)_{i=1}^n$ where we have in order the
position, mass, velocity, and radius of the $i$:th particle, and
$F_{i}$ is the force acting on that particle.  Denote the set of all
states by $E$. Suppose the \emph{physics} of this (discrete-time)
system is given by the function $\beta\colon E\to E$.  The forces
$F_i$ are completely determined by the positions and the masses of the
particles and are given by $\beta$.  Let us now assume that one of the
particles is in fact \emph{the agent}. W.l.o.g. suppose it is the
particle corresponding to $i=1$. Suppose the current state of the
environment is $e=(p_i,m_i,v_i,r_i,F_i)_{i=1}^n.$ Then the next state
is given by $e'=\beta(p_i,m_i,v_i,r_i,F'_i)_{i=1}^n$ where $F'_i=F_i$
if $i>1$ and for $i=1$ we set $F_i'=F_i+F$ where $F$ is the force
applied by the agent at that time. The agent, in turn, is modelled as
being in a state $a=(b,F)$ where $b$ is a hidden variable (``the
brain''), and $F$ is the applied force. Let $A$ be the set of all such
states.  The internal dynamics of the agent is given by some function
$\alpha\colon A\to A$. If the agent is in a state $(b,F)$, the next
state will be given by $\alpha(b,F+F_1)$ where $F_1$ is the force
which is externally applied to the agent's body at that time. This
agent might be described as ``sensing'' the environment's applied
force because, at least in principle, its brain state can depend on
it. This leads to coupled dynamics of the set of states of the form
$(b,F,p_i,m_i,v_i,r_i)$ with a transition function we will denote
by $\alpha*\beta$.

\textbf{Why is this important?} %
This toy example, of course, has nothing new to it in terms of
mathematics or physics. But it neatly separates the dynamics of the
embodied agent from the dynamics of the environment while
simultaneously coupling those two through the addition of the forces.
Changing the embodiment of the agent could mean, for example, that the
agent consists of two particles instead of one, or a particle with a
different radius. This does not influence the underlying physics of
the environment which can still be presented using $\beta$. On the other
hand the same agent can now be placed in a different environment
(different number of particles, different masses, or even different
physics~$\beta'$), as long as the basic principle of additive forces
remains. In fact, even the
principle of additive forces can change.  Just replace the ``$+$''
sign with another operation. The agent can still be presented as
$(B,\alpha)$. The key difference in this example to the classical model
is that the environment is presented \emph{parametrically} with an
underlying set of \emph{rules} on which the dynamics is based instead
of declaring the dynamics as an abstract entity. This shift is
  akin to moving from abstract functions to formulas which is a
familiar trick to the logicians. The formula $a+a+b$ can be
\emph{evaluated} in any group $(G,+)$ once $a,b\in G$ are
given. Compare this to the situation where we abstractly give a
function $f\colon G\times G\to G$ for some group $G$. Then, it is not
clear how this function ``generalizes'' to another group $G'$. This is
the \textbf{power} of $G$-systems as compared to the classical
approach described above.

\section{$G$-systems and coupling}

A \emph{magma} is a pair $(G,\bullet)$ if
$\bullet\colon G\times G\to G$ is any funtion. Typically for us,
$(G,\bullet)$ will be either an abelian group or a lattice with
$\bullet$ representing either the meet operation (the greatest lower
bound) or the join. Given a magma $G=(G,\bullet)$ and a set $X$, a
\emph{$G$-dynamical system}, or $G$-system for short, is a triple
$(G,X,\alpha)$ where $\alpha\colon G^X\to G^X$ is the \emph{transition
  function} and $G^X$ is the product space consisting of all functions
from $X$ to~$G$. This is clearly a very general model of a
discrete-time dynamical system.  One might ask, what about continuous
time systems which are, after all, ubiquotous both in the modelling of
E-cognition and robotics? Most qualitative concepts such as
attractors, decompositions, couplings, and, as we will see, enactivist
concepts such as attunement and affordances can be formulated in the
discrete-time scenario in a way which is easily generalizable to
continuous time systems. The upside of discrete systems, however, is
that one can avoid (for a longer time) talking about measures,
$\sigma$-algebras, thinking about which norm or metric to choose
etc.~\cite{Weinstein_EMS}. We will also see that the discrete-time
system is more suitable for logical analysis. As a result we have
chosen to study discrete-time systems in SYSCOG, because we see them
as a better platform for formulating most concepts and theorems before
generalizing them to the continuous setting rather than vice versa.

If $G$ is clear from the context or fixed, drop it from the notation
and consider $G$-systems as paris $(X,\alpha)$. Let $X$, $Y$, and $Z$
be such that $X\cup Y=Z$, and let $\XX=(X,\alpha)$ and
$\YY=(Y,\beta)$ be $G$-systems.  Then $\ZZ=(Z,\gamma)$ is the
\emph{coupling} between $\XX$ and $\YY$, denoted $\ZZ=\XX*\YY$, if
$\gamma\colon G^Z\to G^Z$ satisfies the equation
$$\gamma(\gg)(z)=
\begin{cases}
  \alpha(\gg)(z),&\text{ if }z\in X\setminus Y\\
  \alpha(\gg)(z)\bullet\beta(\gg)(z),&\text{ if }z\in X\cap Y\\
  \beta(\gg)(z),&\text{ if }z\in Y\setminus X.
\end{cases}
$$
In this case we also denote
$\gamma=\alpha*\beta$. 
If $(G,\bullet)$ is associative, then so is $*$, and if $\bullet$ is
commutative, then so is~$*$.  If $\bullet$ is the operation on $G$,
abusing notation, use the same symbol to denote the pairwise operation
on $G^X$ for any~$X$.  We have the following observation (our first
``result''):

\begin{Lemma}
  Suppose that $(G,\bullet)$ is associative and commutative.
  Let $X,Y$ be some sets, $\alpha,\alpha'\colon G^X\to G^X$,
  and $\beta,\beta'\colon G^Y\to G^Y$. Then
  $$(\alpha\bullet \alpha')*(\beta\bullet\beta')=(\alpha*\beta)\bullet(\alpha'*\beta').$$
\end{Lemma}
\begin{proof}
  Checking the definitions.
\end{proof}

\begin{Def}
  If $A$ is a set, $f\colon A\to A$ is a function, and
  $C\subset A$ satisfies $f[C]\subseteq C$, we say
  that $C$ is \emph{closed under $f$}.
  If $D$ is another set and
  $f\colon A\times D\to A$, $C\subset A$ is \emph{closed
  under} $f$, if $f[C\times D]\subset C$.
\end{Def}

\begin{Def}
  Suppose $A\subset B$. Then $G^A$ acts on $G^B$ by tranlation
  $g\cdot h\mapsto g+h$. 
\end{Def}

Fix a magma $(G,\bullet)$.

\begin{Def}
  Suppose $\XX=(X,\alpha)$ and $\YY=(Y,\beta)$ are $G$-systems. If
  $H_0\subset G^X$ and $H_1\subset G^Y$ are any subsets, we denote by
  $H_0*H_1$ the set
  $$\{h\in G^{X\cup Y}\mid h\rest_X\in H_0\land h\rest_Y\in H_1\}.$$
\end{Def}

\begin{Lemma}\label{lemma:closure}
  Suppose $\XX=(X,\alpha)$ and $\YY=(Y,\beta)$ are $G$-dynamical
  systems. Suppose $H_0\subset G^X$ and $H_1\subset G^Y$ are closed
  under $\alpha$ and $\beta$ respectively. Also assume one of the
  following:
  \begin{enumerate}
  \item $H_0$ is closed under the function
    $\hh\mapsto \hh + \beta(\hh')\rest X$ for all $\hh'\in H_1$, and
    $H_1$ is closed under the function
    $\hh'\mapsto \hh' + \beta(\hh)\rest Y$ for all $\hh\in H_0$.
    \label{lemma:closure_cond1}
  \item $H_0$ and $H_1$ are closed under the translation actions of
    $G^{X\cap Y}$ on $G^X$ and $G^Y$ respectively.
  \end{enumerate}
  Then the
  set $H:=H_0*H_1$ is closed under $\gamma=\alpha*\beta$.
\end{Lemma}
\begin{proof}
  Suppose $\hh\in H$. Then there are $\hh_0\in H_0$ and $\hh_1\in H_1$
  such that $\hh\rest X=\hh_0$ and $\hh\rest Y=\hh_1$.  So
  $\gamma(\hh)=\alpha(\hh_0)+\beta(\hh_1)$.  By the closure
  assumptions, $\alpha(\hh_0)$ and $\beta(\hh_1)$ are in $H_0$ and
  $H_1$ respectively. By the closure under the action of
  $G^{X\cap Y}$, we have
  $\alpha(\hh_0)+\beta(\hh_1)\rest_{(X\cap Y)}\in H_0$.  But since
  $\sprt(\beta(\hh_1))\subset Y$, we have
  $\alpha(\hh_0)+\beta(\hh_1)\rest_{(X\cap
    Y)}=(\alpha(\hh_0)+\beta(\hh_1))\rest_{X}=\gamma(\hh)\rest_X\in
  H_0$. Similarly obtain $\gamma(\hh)\rest_Y\in H_1$. Thus,
  $\gamma(\hh)\in H$, as required.
\end{proof}

\begin{Def}
  As a corollary to the above lemma, we can consider systems of the
  form $(H,\alpha)$ where $H\subset G^X$ is closed under $\alpha$.  In
  fact, it is enough to consider a transition function which
  is only defined on $H$, $\alpha\colon H\to H$.  Such systems
  can be coupled under the additional condition of closure under
  $G^{X\cap Y}$ as in the lemma.
\end{Def}

\MyBox{\textbf{Box 2: A more general definition of coupling}.  Let
  $(G,\bullet)$ be a magma. Let $(X,\alpha)$ and $(Y,\beta)$ be
  $G$-systems, and let $\zeta\colon A\to B$ be a bijective ``gluing
  map'' from some $A\subset X$ to some $B\subset Y$. We can then form
  the coupled system $(Z,\gamma)$ such that $Z=X\cup_\zeta Y$ is the
  set $X\sqcup Y/\sim$ where $\sim$ is the finest equivalence relation
  on the disjoint union $X\sqcup Y$ where $x\sim \zeta(x)$ for all
  $x\in A$. Category theoretically $Z$ is the pushout of $X$ and $Y$
  over $A,\id,\zeta$.  Then for a $\sim$-equivalence class $[z]$ of
  $z\in X\sqcup Y$, let
  $$\gamma(\gg)([z])=
  \begin{cases}
    \alpha(\gg)(z),&\text{ if }z\in X\setminus A\\
    \alpha(\gg)(z)\bullet\beta(\gg)(\zeta(z)),&\text{ if }z\in A\\
    \alpha(\gg)(\zeta^{-1}(z))\bullet\beta(\gg)(z),&\text{ if }z\in B\\
    \beta(\gg)(z),&\text{ if }z\in Y\setminus B.
  \end{cases}
  $$
}

\subsection{History $I$-spaces as special cases}

Recall the classical definitions of Section~\ref{ssec:Classical}.  We
can view the coupling of $(X,U,Y,h,f)$ with $(I,\f,U,Y,\pi)$ as a
special case of a coupling of $G$-dynamical systems in a number of
different ways.  We will present one way. First, without loss of
generality, we can assume that there is some fixed group $G$ such that
$X\subset G^{X'}$ for some $X'$ (if nothing else works, let $X'=X$,
$G=\Z/2\Z$ and identify $x\in X$ with the element
$\gg_x\colon X'\to G$ such that $\gg_x(x')=1\iff x'=x$),
$Y\subset G^{Y'}$, $U\subset G^{U'}$, and $\II\subset G^{\II'}$.  Let
$A=X'\sqcup Y'\sqcup U'$ and $B=\II'\sqcup Y'\sqcup U'$.  Moreover we
can assume w.l.o.g. that $0\notin Y$ and $0\notin U$ where by $0$ we
mean the neutral element of $G^{Y'}$ and $G^{U'}$ respectively.  Let
$\hat Y$ and $\hat U$ the sets $Y$ and $U$ to which the respective
neutral elements have been added.  Then
$H_0:=X\times Y\times \hat U\subset G^A$ and
$H_1:=\II\times \hat Y\times U\subset G^B$ and we can define
$\alpha\colon H_0\colon H_0$, $\beta\colon H_1\to H_1$ by
$\alpha(x,y,u)=(f(x,u),h(f(x,u)),0)$ and
$\beta(\iota,y,u)=(\f(\iota,y),0,\pi(\f(\iota,y))).$ Then the coupled
system $(H_0,\alpha)*(H_1,\beta)$ represents exactly the coupled system
of Section~\ref{ssec:Classical}.  Note that $A\cap B=Y\sqcup U$, and that
Condition~\ref{lemma:closure_cond1} of Lemma~\ref{lemma:closure} is
satisfied, so the coupling $(H_0,\alpha)*(H_1,\beta)$ is indeed
well-defined.

\subsection{Reducibility and emergence}

We say that $\{X,Y\}$ is a \emph{cover} of $Z$, if $X\cup Y=Z$.
\begin{Def}
  Let $\ZZ=(Z,\gamma)$ be a $G$-system and let $\{X,Y\}$ be a cover
  of $Z$. We say that $\ZZ$ is \emph{reducible to $\{X,Y\}$}, if there
  exist $\alpha\colon G^X\to G^X$ and $\beta\colon G^Y\to G^Y$ such
  that $\ZZ=\XX*\YY$ where $\XX=(X,\alpha)$ and $\YY=(Y,\beta)$.
  In this case we also say, that $(\XX,\YY)$ is a \emph{decomposition}
  of~$\ZZ$. We migh also say $\gamma$ is \emph{supervenient} or
  \emph{emergent} over $X,Y$, if
  $$\gamma = \gamma_1\circ\cdots\circ\gamma_k$$
  for some $\gamma_1,\ldots,\gamma_k$ which are reducible to $\{X,Y\}$.
\end{Def}


Every reducible system is emergent, but not all emergent systems are
reducible:

\begin{Thm}[Emergent, not reducible]\label{thm:Irreducible}
  There is $\ZZ=(G^Z,\gamma)$ and $X,Y\subset Z$ such that
  $\gamma$ is reducible to $X$, $Y$, but $\gamma^2$ is not
\end{Thm}
\begin{proof}
  Let $Z=\{0,1,2,3\}$, $X=\{0,1,2\}$, $Y=\{1,2,3\}$,
  and $G=\Z/2\Z$. Let $\gamma$ be given by
  $$\gamma(b_0,b_1,b_2,b_3)=(b_0,\max\{b_0,b_2\},b_3,b_3).$$
\end{proof}

\subsection{Interpretation in 4E cognition}

One of the claims of the 4E approaches to cognitive
science is that cognitive systems are ``not input-output devices'',
and that dynamical systems are somehow an example of that. Consider the
quote
\begin{quote}
  \emph{The behaviour of an organism arises from the dynamics of its
    interaction with its world, and from our perspective as external
    observers we can best describe this as the interaction between two
    dynamical systems. Why is an agent in such a coupled set of
    dynamical systems not an input/output device, and hence not
    performing computations to generate appropriate outputs from
    snapshots of its sensory inputs?} \hfill I. Harvey \cite{Harvey}
\end{quote}
The author of \cite{Harvey} proceeds to give an answer but does not
ground it in mathematical theorems or mathematical definitions.

In Theorem~\ref{thm:Irreducible} we may interpret $\gamma^2$ as
looking at the dynamics on a higher temproral scale.  Then, this
theorem shows that phenomena observed on a higher temporal scale may
not be reducible to the components from which the low-scale dynamics
emerge. This confirms the intuiion of the enactivists that even though
perception is grounded in sensorimotor feedback between efferent and
afferent neural connections (which is happening on the time scale of
milliseconds), the observed phenomena of perception on the time scale
of hundreds of milliseconds might not be reducible to the
agent-environment dichotomy.

\subsection{Dependence}

The following notion of dependence is inspired by the dependence atom
of dependence logic of \cite{Vaananen2007}:

\begin{Def}\label{def:Dependence}
  Let $(X,\alpha)$ be a $G$-system. Let $A,B\subset X$
  and $\bC\subset G^X$. We say that $A$ \emph{determines} $B$
  over $\bC$, if there are no $\gg_0,\gg_1\in \bC$ such that
  $\gg_0\rest_A=\gg_1\rest_A$ and
  $\alpha(\gg_0)\rest_B\ne \alpha(\gg_1)\rest_B$.
\end{Def}

It is possible to translate this to the language of team
semantics. The set $\bC$ can be viewed as a team with columns
corresponding to elements of $X$ and rows corresponding to elements
of~$\bC$.  The set $f[\bC]$ can be similarly seen as a team in the
same way. We can horizontally concatenate them so that the columns
correspond to the elements of the dijoint union $X_0\sqcup X_1$ where
$X_0=X_1=X$ and each row $\gg\in \bC$ is continued by the row
$\alpha(\gg)$. If $A$ and $B$ are finite, then the determination
defined above can be expressed in the language of dependence logic as
$\dep{\bar a;\bar b}$ where $\bar a$ lists all elements of
$A\subset X_0$ and $\bar b$ lists all elements of $B\subset X_1$.  We
can also translate it differently so that the finiteness assumption is
not needed. Simply form a team with two columns corresponding to $A$
and $B$. Each row is of the form $(\gg\rest_A,\alpha(\gg)\rest_B)$ and
the dependence is then expressed as $\dep{A;B}$.

\begin{Notation}
  If $A$ determines $B$ over $\bC$ in $\XX=(X,\alpha)$, we denote this
  by $\XX\models_{\bC} \dep{A;B}$. If $\bC=G^X$, then we just denote
  it by $\XX\models \dep{A;B}$.
\end{Notation}



Now we will prove the decomposition theorem which justifies
the above definitions. When $G$ has a neutral element $0\in G$,
and $X\subset Z$, then identify an element $\gg\in G^X$ with the element
$\gg'\in G^Z$ with the property that
$\gg'(z)=\gg(z)$ for all $z\in X$ and $\gg'(z)=0$ otherwise. Note that
then we can rewrite the definition coupling between $\alpha\colon G^X\to G^X$
and $\beta\colon G^Y\to G^Y$ as
$(\alpha*\beta)(\gg) = \alpha(\gg\rest_X)\bullet \beta(\gg\rest_Y)$.
where the latter $\bullet$ is applied pointwise in $G^{X\cup Y}$.

\begin{Thm}[Characterization of reducibility]
  Let $(G,+)$ be an abelian group.  Suppose $\ZZ=(Z,\gamma)$ is a
  $G$-system. Suppose $\{X,Y\}$ is a cover of $Z$. Then the
  following are equivalent:
  \begin{enumerate}
  \item $\ZZ$ is reducible to $\{X,Y\}$.
    \label{thm:Decomp_c1}
  \item \label{thm:Decomp_c2} 
    For all $\gg_0,\gg_1,\gg_0',\gg_1'\in G^Z$ satisfying
    \begin{equation}
      \label{eq:thm_decomp_c2_cond}
      \gg_0\rest_Y=\gg'_0\rest_Y,\ \ \gg_1\rest_Y=\gg_1'\rest_Y,\ \ \gg_0\rest_X=\gg_1\rest_X,\ \ \text{and}\ \ \gg_0'\rest_X=\gg'_1\rest_X
    \end{equation}
    the following hold:
    \begin{align}
      \label{eq:thm_decomp_c2_dependencies1}
      & \gamma(\gg_0)=\gamma(\gg_1)\ \Longrightarrow\ \gamma(\gg_0')=\gamma(\gg_1')\\
      \label{eq:thm_decomp_c2_dependencies2}
      & \gamma(\gg_0)=\gamma(\gg'_0)\ \Longrightarrow\ \gamma(\gg_1)=\gamma(\gg_1')
    \end{align}
  \item \label{thm:Decomp_c3}
    For all $\gg_0,\gg_1,\gg_0',\gg_1'\in G^Z$ satisfying
    \eqref{eq:thm_decomp_c2_cond} the following hold:
    \begin{align}
      \label{eq:thm_decomp_c2_dependencies1_}
      & \ZZ\models_{\{\gg_0,\gg_1\}}\dep{X;X\cap Y}\ \Longrightarrow\ \ZZ\models_{\gg_0',\gg_1'}\dep{X;X\cap Y}\\
      \label{eq:thm_decomp_c2_dependencies2_}
      & \ZZ\models_{\{\gg_0,\gg_0'\}}\dep{Y;X\cap Y}\ \Longrightarrow\ \ZZ\models_{\gg_1,\gg_1'}\dep{Y;X\cap Y}
    \end{align}
  \end{enumerate}
\end{Thm}
\begin{proof}
  TBA
\end{proof}

\section{Intervention and causality}

These considerations open doors also to the study of causality in
$G$-systems. This is expected to be an important part of cognitive modelling. A key
notion is that of \emph{intervention} and it can be defined directly
using the following substitution notation: Given $\gg\in G^X$ and
$\ss\in G^S$ for some $S\subset X$, define
\begin{equation}
  \label{eq:Substitution}
  \gg[\ss]=\gg[\ss/S]\in G^X\qquad\text{ as }\qquad \gg[\ss/S](x)=
  \begin{cases}
    \gg(x),&\text{ if } x\notin S\\
    \ss(x),&\text{ if } x\in S.
  \end{cases}
\end{equation}
We can use~\eqref{eq:Substitution} to replace the ``$\mathtt{do}$''
operator of \cite{Pearl2009} and define:
\begin{Def}[Causality]
  Let $\XX=(G^X,\alpha)$ be a $G$-system. Let $A,B\subset X$, and
  $\bC\subset G^X$. We say that $A$ \emph{causally influences $B$ in $\bC$},
  denoted $\XX\models_{\bC} c(A,B)$, if for all $\gg\in \bC$ there
  is $\ss\in G^A$ such that
  $\alpha(\gg)\rest_B\ne \alpha(\gg[\ss/A])\rest_B$.
\end{Def}
This enables one define, for example, when an agent has a causal
influence on environmental variables and vice versa. The potential
challenge is that according to E-cognition agent and environment form
closed causal loops each influencing each other while most (but not
all) of the theory of causality is centered around directed
\emph{acyclic} graphs. See Question set~1.










\bibliographystyle{plain}
\bibliography{processed_ref}

\end{document}